\documentclass{icmart}
\usepackage{amsbsy}
\usepackage{amstext}
\usepackage{amssymb}
\usepackage{amsmath}
\usepackage{dcolumn, graphicx, booktabs}
\usepackage{url}
\newtheorem{proposition}{Proposition}

\newtheorem{example}{Example}
\newcommand{\svskip}{\vspace{1.75mm}}

\def\tr{\mathop{\rm tr}\nolimits}

\def\dist{\mathop{\rm dist}\nolimits}

\def\cond{\mathop{\rm cond}\nolimits}

\def\prox{\mathop{\rm prox}\nolimits}
\def\argmin{\mathop{\rm argmin}\nolimits}
\def\ri{\mathop{\rm ri}\nolimits}

\newcommand{\bb}{\boldsymbol{b}}
\newcommand{\bc}{\boldsymbol{c}}

\newcommand{\bp}{\boldsymbol{p}}

\newcommand{\bu}{\boldsymbol{u}}
\newcommand{\bv}{\boldsymbol{v}}

\newcommand{\bx}{\boldsymbol{x}}
\newcommand{\by}{\boldsymbol{y}}
\newcommand{\bz}{\boldsymbol{z}}
\newcommand{\bA}{\boldsymbol{A}}

\newcommand{\bD}{\boldsymbol{D}}
\newcommand{\bE}{\boldsymbol{E}}

\newcommand{\bI}{\boldsymbol{I}}

\newcommand{\bL}{\boldsymbol{L}}
\newcommand{\bM}{\boldsymbol{M}}

\newcommand{\bS}{\boldsymbol{S}}

\newcommand{\bU}{\boldsymbol{U}}

\newcommand{\bW}{\boldsymbol{W}}
\newcommand{\bX}{\boldsymbol{X}}
\newcommand{\bY}{\boldsymbol{Y}}
\newcommand{\bZ}{\boldsymbol{Z}}

\newcommand{\bbeta}{\boldsymbol{\beta}}

\newcommand{\bDelta}{\boldsymbol{\Delta}}
\newcommand{\bTheta}{\boldsymbol{\Theta}}

\title{The Proximal Distance Algorithm}

\author{Kenneth Lange and Kevin L.\ Keys} 
\contact[klange@ucla.edu]{Kenneth Lange, University of California, Los Angeles} 
\contact[klkeys@ucla.edu]{Kevin L. Keys, University of California, Los Angeles}
\begin{document}
\maketitle

\begin{abstract}
\noindent
The MM principle is a device for creating optimization algorithms satisfying the ascent or descent property. The current survey emphasizes the role of the MM principle in nonlinear programming.
For smooth functions, one can construct an adaptive interior point method based on scaled Bregman barriers. This algorithm does not follow the central path. For convex programming subject to nonsmooth constraints, one can combine an exact penalty method with distance majorization to create versatile algorithms that are effective even in discrete optimization. These proximal distance algorithms are highly modular and reduce to set projections and proximal mappings, both very well-understood techniques in optimization. We illustrate the possibilities in linear programming, binary piecewise-linear programming, nonnegative quadratic programming, $\ell_0$ regression,  matrix completion, and inverse sparse covariance estimation.


\end{abstract}

\section{Introduction}

The MM principle is a device for constructing optimization algorithms \cite{borg07,hunter04,lange00,lange10,lange13}. In essence,
it replaces the objective function $f(\bx)$ by a simpler surrogate function $g(\bx \mid \bx_n)$
anchored at the current iterate $\bx_n$ and majorizing or minorizing $f(\bx)$. As a byproduct of optimizing $g(\bx \mid \bx_n)$ with 
respect to $\bx$, the objective function $f(\bx)$ is sent downhill or uphill, depending on whether the purpose is minimization or 
maximization. The next iterate $\bx_{n+1}$ is chosen to optimize the surrogate $g(\bx \mid \bx_n)$ subject to any relevant constraints.  
Majorization combines two conditions: the tangency condition  $g(\bx_n \mid \bx_n) =  f(\bx_n)$ and the domination condition
$g(\bx \mid \bx_n)  \geq f(\bx)$ for all $\bx$.  In minimization these conditions and the definition of $\bx_{n+1}$ lead to the descent property
\begin{align*}
f(\bx_{n+1}) & \leq  g(\bx_{n+1} \mid \bx_{n}) \leq g(\bx_{n} \mid \bx_{n}) = f(\bx_{n}).
\end{align*}
Minorization reverses the domination inequality and produces an ascent algorithm. Under 
appropriate regularity conditions, an MM algorithm is guaranteed to converge to a stationary point of the objective function \cite{lange13}. From the perspective of dynamical systems, the objective function serves as a Liapunov 
function for the algorithm map.

The MM principle simplifies optimization by: (a) separating the variables of a problem, (b) avoiding large matrix inversions, (c) linearizing a problem, (d) restoring symmetry, (e) dealing with
equality and inequality constraints gracefully, and (f) turning a nondifferentiable problem 
into a smooth problem. Choosing a tractable surrogate function $g(\bx \mid \bx_n)$ that hugs the objective function $f(\bx)$ as tightly as possible requires experience and skill with 
inequalities. The majorization relation between functions is closed under the formation of sums,
nonnegative products, limits, and composition with an increasing function. Hence, it is possible 
to work piecemeal in majorizing complicated objective functions.

It is impossible to do justice to the complex history of the MM principle in a paragraph.
The celebrated EM (expectation-maximization) principle of computational statistics is
a special case of the MM principle \cite{mclachlan08}. 
Specific MM and EM algorithms appeared years before the principle was well understood
\cite{hartley58,mckendrick26,smith57,weiszfeld37,yates34}. The widely applied
projected gradient and proximal gradient algorithms can be motivated from the
MM perspective, but the early emphasis on operators and fixed points obscured
this distinction. Although Dempster, Laird, and Rubin \cite{dempster77} formally named 
the EM algorithm, many of their contributions were anticipated by Baum \cite{baum72} 
and Sundberg \cite{sundberg76}.  The MM principle was clearly stated by 
Ortega and Rheinboldt \cite{ortega00}. de Leeuw \cite{deleeuw77} is generally credited with recognizing the 
importance of the principle in practice.  The EM algorithm had an immediate 
and large impact in computational statistics. The more general
MM principle was much slower to take hold. The papers \cite{deleeuw90,heiser95,kiers90}
by the Dutch school of psychometricians solidified its position. (In this early
literature the MM principle is called iterative majorization.) The related Dinklebach \cite{dinklebach67} maneuver 
in fractional linear programming also highlighted the importance
of the descent property in algorithm construction.

Before moving on, let us record some notational conventions.
All vectors and matrices appear in boldface. The $^*$ superscript indicates a vector
or matrix transpose. The Euclidean norm of a vector $\bx$ is denoted by $\|\bx\|$ and
the Frobenius norm of a matrix $\bM$ by $\|\bM\|_F$. For a smooth real-valued function 
$f(\bx)$, we write its gradient (column vector of partial derivatives) as $\nabla f(\bx)$, 
its first differential (row vector of derivatives) as $df(\bx)=\nabla f(\bx)^*$, and its second differential (Hessian matrix) as  $d^2f(\bx)$. 

\section{An Adaptive Barrier Method}

In convex programming it simplifies matters notationally to 
replace a convex inequality constraint $h_j(\bx) \le 0$ 
by the concave constraint $v_j(\bx) = -h_j(\bx) \ge 0$.
Barrier methods operate on the relative interior of the feasible region
where all $v_j(\bx)>0$. Adding an appropriate barrier term to the 
objective function $f(\bx)$ keeps an initially inactive constraint 
$v_j(\bx)$ inactive throughout an optimization search. If the barrier 
function is well designed, it should adapt and permit convergence to a 
feasible point $\by$ with one or more inequality constraints active. 

We now briefly summarize an adaptive barrier method that does not follow
the central path \cite{lange94}. 
Because the logarithm of a concave function is concave, the Bregman majorization \cite{bregman67} 
\begin{align*}
-\ln v_j(\bx) + \ln v_j(\bx_n) + \frac{1}{v_j(\bx_n)}dv_j(\bx_n)(\bx-\bx_n) & \ge  0
\end{align*}
acts as a convex barrier for a smooth constraint $v_j(\bx) \ge 0$.
To make the barrier adaptive, we scale it by the current value
$v_j(\bx_n)$ of the constraint. These considerations suggest an
MM algorithm based on the surrogate function
\begin{align*}
g(\bx \mid\bx_n) & =  f(\bx) - \rho \sum_{j=1}^s v_j(\bx_n) \ln v_j(\bx)
+ \rho \sum_{j=1}^s dv_j(\bx_n)(\bx-\bx_n) 
\end{align*}
for $s$ inequality constraints. Minimizing the surrogate subject to 
relevant linear equality constraints $\bA\bx=\bb$
produces the next iterate $\bx_{n+1}$. The constant $\rho$ determines the 
tradeoff between keeping the constraints inactive and minimizing $f(\bx)$.
One can show that the MM algorithm with exact minimization converges to 
the constrained minimum of $f(\bx)$ \cite{lange13}.

In practice one step of Newton's method is usually adequate to
decrease $f(\bx)$. The first step of Newton's method minimizes
the second-order Taylor expansion of $g(\bx \mid \bx_n)$ around $\bx_n$ subject
to the equality constraints. Given smooth functions, the two differentials 
\begin{align}
dg(\bx_n \mid \bx_n)  & = df(\bx_n) \nonumber \\
d^2g(\bx_n \mid \bx_n) &  = d^2f(\bx_n) - \rho \sum_{j=1}^s d^2v_j(\bx_n) 
\label{interior_pt_derivatives} \\
 & + \rho \sum_{j=1}^s \frac{1}{v_j(\bx_n)}\nabla v_j(\bx_n) dv_j(\bx_n) 
\nonumber 
\end{align}
are the core ingredients in the quadratic approximation of $g(\bx \mid \bx_n)$.
Unfortunately, one step of Newton's method is neither guaranteed to 
decrease $f(\bx)$ nor to respect the nonnegativity constraints.

\begin{example}
Adaptive Barrier Method for Linear Programming
\end{example}
For instance, the standard form of linear programming requires minimizing a 
linear function $f(\bx) = \bc^*\bx$ subject to $\bA \bx = \bb$ and $\bx \ge {\bf 0}$.
The quadratic approximation to the surrogate $g(\bx \mid \bx_n)$ amounts to
\begin{align*}
\bc^*\bx_n+\bc^*(\bx-\bx_n)+ \frac{\rho}{2} \sum_{j=1}^p 
\frac{1}{x_{nj}}(x_j-x_{nj})^2 .
\end{align*}
The minimum of this quadratic subject to the linear equality constraints occurs at the point 
\begin{align*}
\bx_{n+1} & = \bx_n -\bD_n^{-1} \bc+\bD_n^{-1}\bA^* 
(\bA \bD_n^{-1} \bA^*)^{-1}(\bb-\bA \bx_n+\bA \bD_n^{-1}\bc).
\end{align*}
Here $\bD_n$ is the diagonal matrix with $i$th diagonal entry 
$\rho x_{ni}^{-1}$, and the increment $\bx_{n+1}-\bx_n$ satisfies the
linear equality constraint $\bA(\bx_{n+1}-\bx_n) = \bb-\bA\bx_n$.
\qed \svskip

One can overcome the objections to Newton updates by taking a controlled 
step along the Newton direction $\bu_n=\bx_{n+1}-\bx_n$.  The key is to
exploit the theory of self-concordant functions \cite{boyd04,nesterov94}.  
A thrice differentiable convex function $h(t)$ is said to be self-concordant 
if it satisfies the inequality
\begin{align*}
\left| h'''(t) \right| & \le 2c h''(t)^{3/2} 
\end{align*}
for some constant $c \ge 0$ and all $t$ in the essential domain of $h(t)$. 
All convex quadratic functions qualify as self-concordant with $c=0$. The function 
$h(t) = -\ln (at+b)$ is self-concordant with constant 1.
The class of self-concordant functions is closed under sums and composition with linear functions. A convex function $k(\bx)$ 
with domain $\mathsf{R}^p$ is said to be self-concordant if every slice 
$h(t) = k(\bx+t\bu)$ is self-concordant.

Rather than conduct an expensive one-dimensional search along the Newton direc\-tion 
$\bx_n+t\bu_n$, one can majorize the surrogate function $h(t) = g(\bx_n+t\bu_n \mid \bx_n)$ 
along the half-line $t \ge 0$.  The clever majorization 
\begin{align}
h(t) & \le h(0)+h'(0)t-\frac{1}{c}h''(0)^{1/2}t - \frac{1}{c^2} \ln [1-cth''(0)^{1/2}]
\label{self_concordant_majorization}
\end{align}
serves the dual purpose of guaranteeing a decrease in $f(\bx)$ and
preventing a violation of the inequality constraints \cite{nesterov94}. Here $c$ is the 
self-concordance constant associated with the surrogate. The optimal choice of $t$ reduces to the damped Newton update
\begin{align}
t & = \frac{h'(0)}{h''(0)-ch'(0) h''(0)^{1/2}} . \label{damped_newton_step} 
\end{align}
The first two derivatives of $h(t)$ are clearly
\begin{align*}
h'(0) &  =  df(\bx_n)\bu_n \\
h''(0) & =  \bu_n^*d^2f(\bx_n)\bu_n - \rho \sum_{j=1}^s \bu_n^* d^2v_j(\bx_n)\bu_n \\
& + \rho \sum_{j=1}^s \frac{1}{v_j(\bx_n)}[dv_j(\bx_n)\bu_n]^2 .
\end{align*}
The first of these derivatives is nonpositive because $\bu_n$ is a descent direction for $f(\bx)$.
The second is generally positive because all of the contributing terms are nonnegative.

\begin{table}[tbh]
\vspace{.1in}
\begin{center}
\begin{tabular}{cccccc} \hline
& \multicolumn{2}{c}{No Safeguard} & \multicolumn{3}{c}{Self-concordant Safeguard} \\ 
\cmidrule(r){2-3} \cmidrule(r){4-6} 
Iteration $n$ & $\bc^*\bx_n$  &  $\|\bDelta_n\|$ & $\bc^*\bx_n$  & $\|\bDelta_n\|$  
&  $t_n$  \\ \hline
   1 &   -1.20000 &    0.25820 &    -1.11270 &    0.14550 &    0.56351 \\
   2 &   -1.33333 &    0.17213 &    -1.20437 &    0.11835 &    0.55578 \\
   3 &   -1.41176 &    0.10125 &    -1.27682 &    0.09353 &    0.55026 \\
   4 &   -1.45455 &    0.05523 &    -1.33288 &    0.07238 &    0.54630 \\
   5 &   -1.47692 &    0.02889 &    -1.37561 &    0.05517 &    0.54345 \\
  10 &   -1.49927 &    0.00094 &    -1.47289 &    0.01264 &    0.53746 \\
  15 &   -1.49998 &    0.00003 &    -1.49426 &    0.00271 &    0.53622 \\
  20 &   -1.50000 &    0.00000 &    -1.49879 &    0.00057 &    0.53597 \\
  25 &   -1.50000 &    0.00000 &    -1.49975 &    0.00012 &    0.53591 \\
  30 &   -1.50000 &    0.00000 &    -1.49995 &    0.00003 &    0.53590 \\
  35 &   -1.50000 &    0.00000 &    -1.49999 &    0.00001 &    0.53590 \\
  40 &   -1.50000 &    0.00000 &    -1.50000 &    0.00000 &    0.53590 \\
\hline
\end{tabular}
\end{center}
\caption{Performance of the adaptive barrier method in linear programming. \label{table0}}
\end{table}

When $f(\bx)$ is quadratic and the inequality constraints are affine, detailed
calculations show that the surrogate function $g(\bx \mid \bx_n)$ is self-concordant with constant 
\begin{align*}
c &  = \frac{1}{\sqrt{\rho \min\{v_1(\bx_n),\ldots,v_s(\bx_n)\}}}.
\end{align*}
Taking the damped Newton's step with step length (\ref{damped_newton_step})
keeps $\bx_n+t_n\bu_n$ in the relative interior 
of the feasible region while decreasing the surrogate and hence the objective function
$f(\bx)$. When $f(\bx)$ is not quadratic but can be majorized by a quadratic $q(\bx \mid \bx_n)$, 
one can replace $f(\bx)$ by $q(\bx \mid \bx_n)$ in calculating the adaptive-barrier update.
The next iterate $\bx_{n+1}$ retains the descent property. 

As a toy example consider the linear programming problem of minimizing $\bc^*\bx$
subject to $\bA \bx = \bb$ and $\bx \ge {\bf 0}$. Applying the adaptive barrier 
method to the choices
\begin{align*}
\bA & =  
\begin{pmatrix} 2 & 0 & 0 & 1 & 0 & 0 \\
0 & 2 & 0 & 0 & 1 & 0 \\
0 & 0 & 2 & 0 & 0 & 1
\end{pmatrix}, \quad \bb  = 
\begin{pmatrix} 1 \\ 1 \\ 1
\end{pmatrix}, \quad 
\bc  = 
\begin{pmatrix} -1 \\ -1 \\ -1 \\ 0 \\ 0 \\ 0
\end{pmatrix}  
\end{align*}
and to the feasible initial point $\bx_0 = \frac{1}{3}{\bf 1}$ produces the results  
displayed in Table \ref{table0}. Not shown is the minimum point
$(\frac{1}{2},\frac{1}{2},\frac{1}{2},0,0,0)^*$. Columns two and 
three of the table record the progress
of the unadorned adaptive barrier method. The quantity $\|\bDelta_n\|$ equals the
Euclidean norm of the difference vector $\bDelta_n= \bx_n-\bx_{n-1}$. Columns 
four and five repeat this information for the algorithm modified by the self-concordant 
majorization (\ref{self_concordant_majorization}). The quantity $t_n$ in column six represents 
the optimal step length (\ref{damped_newton_step}) in going from $\bx_{n-1}$ to $\bx_n$
along the Newton direction $\bu_{n-1}$. Clearly, there is a price to be paid in implementing 
a safeguarded Newton step. In practice, this price is well worth paying.

\section{Distance Majorization}

On a Euclidean space, the distance to a closed set $S$ is a Lipschitz 
function  $\dist(\bx,S)$ with Lipschitz constant 1.  If $S$ is also convex, then
$\dist(\bx,S)$ is a convex function.  Projection onto $S$ is intimately tied
to $\dist(\bx,S)$. Unless $S$ is convex, the projection operator
$P_S(\bx)$ is multi-valued for at least one argument $\bx$.  Fortunately, it 
is possible to majorize $\dist(\bx,S)$ at $\bx_n$ by $\|\bx-P_S(\bx_n)\|$.  This simple
observation is the key to the proximal distance algorithm to be discussed
later.  In the meantime, let us show how to derive two feasibility algorithms by
distance majorization \cite{chi13}.  Let $S_1,\ldots,S_m$ be closed sets. The method
of averaged projections attempts to find a point in their intersection 
$S = \cap_{j=1}^m S_j$. To derive the algorithm, consider the convex combination
\begin{align*}
f(\bx) & = \sum_{j=1}^m \alpha_j \dist(\bx,S_j)^2
\end{align*}
of squared distance functions. Obviously, $f(\bx)$ vanishes precisely
on $S$ when all $\alpha_j>0$. The majorization
\begin{align*}
g(\bx \mid \bx_n) & = \sum_{j=1}^m \alpha_j \|\bx-P_{S_j}(\bx_n)\|^2
\end{align*}
of $f(\bx)$ is easily minimized. The minimum point of $g(\bx \mid \bx_n)$,
\begin{align*}
\bx_{n+1} & = \sum_{j=1}^m \alpha_j P_{S_j}(\bx_n),
\end{align*}
defines the averaged operator. The MM principle guarantees that $\bx_{n+1}$ 
decreases the objective function.

Von Neumann's method of alternating projections can also be derived from this
perspective. For two sets $S_1$ and $S_2$, consider the problem of minimizing
the objective function $f(\bx) = \dist(\bx,S_2)^2$ subject to the constraint 
$\bx \in S_1$. The function
\begin{align*}
g(\bx \mid \bx_n) & = \|\bx-P_{S_2} (\bx_n)\|^2
\end{align*} 
majorizes $f(\bx)$. Indeed, the domination condition $g(\bx \mid \bx_n) \ge f(\bx)$
holds because $P_{S_2} (\bx_n)$ belongs to $S_2$; the tangency condition 
$g(\bx_n \mid \bx_n) = f(\bx_n)$ holds because $P_{S_2}(\bx_n)$ is the closest
point in $S_2$ to $\bx_n$. The surrogate function $g(\bx \mid \bx_n)$ 
is minimized subject to the constraint by taking 
$\bx_{n+1}= P_{S_1} \circ P_{S_2}(\bx_n)$. The MM principle again ensures that 
$\bx_{n+1}$ decreases the objective function. When the two sets intersect, the
least distance of 0 is achieved at any point in the intersection. One can
extend this derivation to three sets by minimizing the objective function
$f(\bx)=\dist(\bx,S_2)^2+\dist(\bx,S_3)^2$ subject to $\bx \in S_1$.
The surrogate
\begin{align*}
g(\bx \mid \bx_n) & = \|\bx-P_{S_2}(\bx_n)\|^2+\|\bx-P_{S_3}(\bx_n)\|^2 \\
& = 2 \Big\|\bx-\frac{1}{2}[P_{S_2}(\bx_n)+P_{S_3}(\bx_n)]\Big\|^2+c_n
\end{align*}
relies on an irrelevant constant $c_n$. The closest point in $S_1$ is 
\begin{align*}
\bx_{n+1} & = P_{S_1}\left\{\frac{1}{2}[P_{S_2}(\bx_n)+P_{S_3}(\bx_n)] \right\}.
\end{align*}
This construction clearly generalizes to more than three sets.

\section{The Proximal Distance Method}

We now turn to an exact penalty method that applies to nonsmooth functions. 
Clarke's exact penalty method \cite{clarke83} turns the constrained problem 
of minimizing a function $f(\by)$ over a closed set $S$ into the unconstrained 
problem of minimizing the function $f(\by)+\rho \dist(\by,S)$ for 
$\rho$ sufficiently large. Here is a precise statement of a generalization 
of Clarke's result
\cite{borwein00,clarke83,demyanov10}.
\begin{proposition} \label{proposition1}
Suppose $f(\by)$ achieves a local minimum on $S$ at the point 
$\bx$. Let $\phi_S(\by)$ denote a function that vanishes on $S$ and satisfies
$\phi_S(\by) \ge c \dist(\by,S)$ for all $\bx$ and some positive constant $c$.  
If $f(\by)$ is locally Lipschitz around $\bx$ with constant $L$, then for every 
$\rho \ge c^{-1}L$, $F_\rho(\by)=f(\by) + \rho \phi_S(\by)$ achieves a local 
unconstrained minimum at $\bx$.
\end{proposition}

Classically the choice $\phi_S(\bx)=\dist(\bx,S)$ was preferred.
For affine equality constraints $g_i(\bx)=0$ and affine inequality 
constraints $h_j(\bx) \le 0$, Hoffman's bound 
\begin{align*}
\dist(\by,S) & \le \tau \left\|\begin{matrix} G(\by) \\ H(\by)_+ \end{matrix} \right\|
\end{align*}
applies, where $\tau$ is some positive constant, $S$ is the feasible set where
$G(\by) = {\bf 0}$, and $H(\by)_+ \le {\bf 0}$ \cite{hoffman52}.  The vector $H(\by)_+$ 
has components $h_j(\bx)_+=\max\{h_j(\by),0\}$. When $S$ is 
the intersection of several closed sets $S_1,\ldots,S_m$, the alternative 
\begin{align}
\phi_S(\by) & = \sqrt{\sum_{i=1}^m \dist(\by,S_i)^2} \label{sum_of_distances}
\end{align}
is attractive. The next proposition gives sufficient conditions under
which the crucial bound $\phi_S(\by) \ge c \dist(\by,S)$  is valid for the function 
(\ref{sum_of_distances}).
\begin{proposition} \label{proposition0}
Suppose $S_1,\ldots,S_m$ are closed convex sets in $\mathsf{R}^p$
with the first $j$ sets polyhedral. Assume further that the intersection
\begin{align*}
S & = (\cap_{i=1}^j S_i) \cap (\cap_{i=j+1}^m \ri\:S_i)
\end{align*}
is nonempty and bounded. Then there exists a constant $\tau>0$ such that
\begin{align*} 
\dist(\bx,S) & \le \tau \sum_{i=1}^m \dist(\bx,S_i) 
 \le \tau \sqrt{m} \sqrt{ \sum_{i=1}^m \dist(\bx,S_i)^2}
\end{align*}
for all $\bx$. The sets $S_1,\ldots,S_m$ are said to be linearly regular.
\end{proposition}
\begin{proof} See the references \cite{bauschke99,deutsch01} for all details. 
A polyhedral set is the nonempty intersection of a finite number of half-spaces. 
The operator $\ri\, K$ forms the relative interior
of the convex set $K$, namely, the interior of $K$ relative to the affine hull 
of $K$. When $K$ is nonempty, its relative interior is nonempty and generates the
same affine hull as $K$ itself. \end{proof} 

In general, we will require $f(\bx)$ and $\phi_S(\bx)$ to be continuous functions
and the sum $F_\rho(\by)=f(\by) + \rho \phi_S(\by)$ to be coercive for some value $\rho=\rho_0$.  
It then follows that $F_\rho(\by)$ is coercive and attains its minimum for all $\rho \ge \rho_0$.
One can prove a partial converse to Clarke's theorem \cite{demyanov10,DemDiFac1998}.
This requires the enlarged set $S_\epsilon = \{\bx: \phi_S(\bx) < \epsilon\}$ 
of points lying close to $S$ as measured by $\phi_S(\bx)$. 
\begin{proposition} \label{proposition2}
Suppose that $f(\by)$ is Lipschitz on $S_\epsilon$ for some $\epsilon > 0$. 
Then under the stated assumptions on $f(\bx)$ and $\phi_S(\bx)$, a 
global minimizer of $F_\rho(\by)$ is a constrained minimizer of $f(\by)$
for all sufficiently large $\rho$. 
\end{proposition}

When the constraint set $S$ is compact and $f(\by)$ has a continuously 
varying local Lipschitz constant, the hypotheses of Proposition \ref{proposition2} 
are fulfilled. This is the case, for instance, when $f(\by)$ is continuously 
differentiable. With this background on the exact penalty method in mind, we now
sketch an approximate MM algorithm for convex programming that is motivated by 
distance majorization. This algorithm is designed to exploit set projections
and proximal maps. The proximal map $\prox_h(\by)$ associated with a convex function $h(\bx)$
satisfies
\begin{align*}
\prox_h(\by) & =  \argmin_{\bx} \Big[h(\bx)+\frac{1}{2}\|\by-\bx\|^2 \Big].
\end{align*}
A huge literature and software base exist for computing projections and proximal
maps \cite{bauschke11}.

Since the function $\dist(\bx,S)$ is merely continuous, we advocate 
approximating it by the differentiable function 
\begin{align*}
\dist_{\epsilon}(\bx,S) & = \sqrt{\dist(\bx,S)^2+\epsilon}
\end{align*}
for $\epsilon>0$ small. The composite function $\dist_{\epsilon}(\bx,S)$ is 
convex when $S$ is convex because the function $\sqrt{t^2+\epsilon}$ is increasing and convex
on $[0,\infty)$. Instead of minimizing $f(\bx)+\rho \dist(\bx,S)$, we 
suggest minimizing the differentiable convex function 
$f(\bx)+\rho \dist_\epsilon(\bx,S)$ by an MM algorithm. Regardless of
whether $S$ is convex, the majorization
\begin{align}
\dist_{\epsilon}(\bx,S) & \le \sqrt{\|\bx-P_S(\bx_n)\|^2+\epsilon}
\label{distance_to_projection_majorization}
\end{align}
holds. If $S$ is nonconvex, there may be a multiplicity of closest points,
and one must choose a representative of the set $P_S(\bx_n)$.
In any event one can invoke the univariate majorization
\begin{align}
\sqrt{t} & \ge \sqrt{t_n}+\frac{t-t_n}{2\sqrt{t_n}}  \label{square_root_majorization}
\end{align}
of the concave function $\sqrt{t}$ on the interval $t>0$ and majorize the majorization (\ref{distance_to_projection_majorization}) by
\begin{align*}
\sqrt{\|\bx-P_S(\bx_n)\|^2+\epsilon} & \le 
\frac{1}{2 \sqrt{\|\bx_n-P_S(\bx_n)\|^2+\epsilon}}\|\bx-P_S(\bx_n)\|^2 +c_n
\end{align*}
for some irrelevant constant $c_n$. The second step of our proposed
MM algorithm consists of minimizing the surrogate function
\begin{align*}
g(\bx \mid \bx_n) & = f(\bx)+ \frac{w_n}{2}\|\bx-P_S(\bx_n)\|^2 \\
w_n & = \frac{\rho}{\sqrt{\|\bx_n-P_S(\bx_n)\|^2+\epsilon}} .
\end{align*}
The corresponding proximal map drives $f(\bx)+\rho \dist_\epsilon(\bx,S)$ downhill.
Under the more general exact penalty (\ref{sum_of_distances}),
the surrogate function depends on a sum of spherical quadratics
rather than a single spherical quadratic.

It is possible to project onto a variety of closed nonconvex sets.
For example, if $S$ is the set of integers, then projection amounts to
rounding. An ambiguous point $n+\frac{1}{2}$ can be projected
to either $n$ or $n+1$. Projection onto a finite set simply
tests each point separately. Projection onto a Cartesian
product is achieved via the Cartesian product of the
projections. One can also project onto many 
continuous sets of interest. For example, to project onto
the closed set of points having at most $k$ nonzero
coordinates, one zeros out all but the $k$ largest coordinates 
in absolute value. Projection onto the sphere of
center $\bz$ and radius $r$ takes $\by \ne \bz$
into the point $\bz+\frac{r}{\|\by-\bz\|}(\by-\bz)$. All
points of the sphere are equidistant from its center. 

By definition the update $\bx_{n+1}=\prox_{w_n^{-1}f}[P_S(\bx_n)]$  
minimizes $g(\bx \mid \bx_n)$. We will refer to this MM algorithm as the 
{\bf proximal distance algorithm}. It enjoys several virtues. First, it allows 
one to exploit the extensive body of results on proximal maps and projections. 
Second, it does not demand that the constraint set $S$ be convex. Third, it does 
not require the objective function $f(\bx)$ to be convex or smooth.
Finally, the minimum values and minimum points of the functions 
$f(\bx)+\rho \dist(\bx,S)$ and
$f(\bx)+\rho \dist_\epsilon(\bx,S)$ are close when $\epsilon>0$ is small.

In implementing the proximal distance algorithm, the constants $L$ and $\epsilon$ must
specified. For many norms the Lipschitz constant $L$ is known. For a differentiable
function $f(\bx)$, the mean value inequality suggests taking $L$ equal to 
the maximal value of $\|\nabla f(\bx)\|$ in a neighborhood of the optimal point. 
In specific problems a priori bounds can be derived.  If no such
prior bound is known, then one has to guess an appropriate $\rho$ and 
see if it leads to a constrained minimum. If not, $\rho$ should be systematically 
increased until a constrained minimum is reached. Even with a
justifiable bound, it is prudent to start $\rho$ well below its
intended upper bound to emphasize minimization of the loss
function in early iterations.  Experience shows that gradually decreasing 
$\epsilon$ is also  a good tactic; otherwise, one again runs the risk of putting too
much early stress on satisfying the constraints. In practice the sequences 
$\rho_n = \min\{\alpha^n \rho_0,\rho_{\max}\}$ and  
$\epsilon_n = \max\{\beta^{-n}\epsilon_0,\epsilon_{\min}\}$ 
work well for $\alpha$ and $\beta$ slightly larger than 1, say 1.2,
and $\rho_0 = \epsilon_0 = 1$.
On many problems more aggressive choices of $\alpha$ and $\beta$ 
are possible. The values of $\rho_{\max}$ and $\epsilon_{\min}$ are
problem specific, but taking $\rho_{\max}$ substantially greater than
a known Lipschitz constant slows convergence. Taking $\epsilon_{\min}$
too large leads to a poor approximate solution. 

\section{Sample Problems}

We now explore some typical applications of the proximal distance algorithm. 
In all cases we are able to establish local Lipschitz constants.
Comparisons with standard optimization software serve as performance
benchmarks.

\begin{example}
Projection onto an Intersection of Closed Convex Sets
\end{example}
Let $S_1,\ldots,S_k$ be closed convex sets, and assume that projection onto
each $S_j$ is straightforward. Dykstra's algorithm \cite{deutsch01,dykstra83}
is designed to find the projection of an external point $\by$ onto 
$S = \cap_{j=1}^k S_j$. The proximal distance algorithm provides an 
alternative based on the convex function
\begin{align*}
f(\bx) & =  \sqrt{\|\bx-\by\|^2+\delta}
\end{align*}
for $\delta$ positive, say $\delta =1$. The choice $f(\bx)$ is preferable to the
obvious choice $\|\bx-\by\|^2$ because $f(\bx)$  is Lipschitz with Lipschitz
constant 1. In the proximal distance algorithm, we take
\begin{align*}
\phi_S(\bx) & =  \sqrt{\sum_{j=1}^k \dist(\bx,S_j)^2}
\end{align*}
and minimize the surrogate function
\begin{align*}
g(\bx \mid \bx_n) & = f(\bx)+\frac{w_n}{2} \sum_{j=1}^k\|\bx-\bp_{nj}\|^2 
 = f(\bx)+\frac{kw_n}{2}\|\bx-\bar{\bp}_{n}\|^2+c_n ,
\end{align*}
where $\bp_{nj}$ is the projection of $\bx_n$ onto $S_j$, $\bar{\bp}_n$
is the average of the projections $\bp_{nj}$, $c_n$ is an irrelevant constant, 
and
\begin{align*}
w_n & = \frac{\rho}{\sqrt{\sum_{j=1}^k\|\bx_n-\bp_{nj}\|^2+\epsilon}} .
\end{align*}
After rearrangement, the stationarity condition for optimality reads 
\begin{align*}
\bx & = (1-\alpha)\by+\alpha \bar{\bp}_n , \quad 
\alpha  = \frac{kw_n}{\frac{1}{\sqrt{\|\bx-\by\|^2+\delta}}+kw_n}.
\end{align*}
In other words, $\bx_{n+1}$ is a convex combination of $\by$ and $\bar{\bp}_n$.

\begin{table}[tbh]
\begin{center}
\begin{tabular}{ccccc}
\toprule
& \multicolumn{2}{c}{Dykstra} & \multicolumn{2}{c}{Proximal Distance} \\ 
\cmidrule(r){2-3} \cmidrule(r){4-5} 
Iteration $n$ & $x_{n1}$ & $x_{n2}$ & $x_{n1}$ & $x_{n2}$ \\ \hline 
  0 & -1.00000 &  2.00000 &     -1.00000 &      2.00000 \\
  1 & -0.44721 &  0.89443 &     -0.44024 &      1.60145 \\
  2 &  0.00000 &  0.89443 &     -0.25794 &      1.38652 \\
  3 & -0.26640 &  0.96386 &     -0.16711 &      1.25271 \\
  4 &  0.00000 &  0.96386 &     -0.11345 &      1.16647 \\
  5 & -0.14175 &  0.98990 &     -0.07891 &      1.11036 \\
 10 &  0.00000 &  0.99934 &     -0.01410 &      1.01576 \\
 15 & -0.00454 &  0.99999 &     -0.00250 &      1.00257 \\
 20 &  0.00000 &  1.00000 &     -0.00044 &      1.00044 \\
 25 & -0.00014 &  1.00000 &     -0.00008 &      1.00008 \\
 30 &  0.00000 &  1.00000 &     -0.00001 &      1.00001 \\
 35 &  0.00000 &  1.00000 &      0.00000 &      1.00000 \\
\bottomrule 
\end{tabular} 
\end{center} 
\caption{Dykstra's algorithm versus the proximal distance algorithm.}
\label{tab:dykstra_versus_mm}
\end{table}

To calculate the optimal coefficient $\alpha$, we minimize the convex surrogate
\begin{align*}
h(\alpha) & = g[(1-\alpha)\by+\alpha \bar{\bp}_n \mid \bx_n] 
= \sqrt{\alpha^2 d^2+\delta}+\frac{kw_n}{2}(1-\alpha)^2d^2+c_n
\end{align*} 
for $d=\|\by-\bar{\bp}_n\|$. Its derivative
\begin{align*}
h'(\alpha) & = \frac{\alpha d^2}{\sqrt{\alpha^2 d^2+\delta}}-kw_n(1-\alpha)d^2
\end{align*}
satisfies $h'(0)<0$ and $h'(1)>0$ and possesses a unique root on the
open interval $(0,1)$. This root can be easily computed by bisection or Newton's method.

Table \ref{tab:dykstra_versus_mm} compares Dykstra's algorithm and the
proximal distance algorithm on a simple planar example. Here $S_1$ is the 
closed unit ball in $\mathsf{R}^2$, and $S_2$ is the closed halfspace with 
$x_1 \ge 0$.  The intersection $S$ reduces to the right half ball centered at the 
origin.  The table records the iterates of the two algorithms from the 
starting point $\bx_0=(-1,2)^*$ until their eventual convergence to 
the geometrically obvious solution $(0,1)^*$. In the proximal distance
method we set $\rho_n = 2$ and aggressively $\epsilon_n = 4^{-n}$. The two algorithms 
exhibit similar performance but take rather different trajectories.
\qed

\begin{example}
Binary Piecewise-Linear Functions
\end{example}
The problem of minimizing the binary piecewise-linear function
\begin{align*}
f(\bx) & = \sum_{i<j} w_{ij}|x_i-x_j|+\bb^*\bx
\end{align*}
subject to $\bx \in \{0,1\}^d$ and nonnegative
weights $w_{ij}$ is a typical discrete optimization
problem with applications in graph cuts. If we invoke the majorization
\begin{align*}
|x_i-x_j| & \le \Big|x_i -\frac{x_{ni}+x_{nj}}{2}\Big|
+\Big|x_j -\frac{x_{ni}+x_{nj}}{2}\Big|
\end{align*}
prior to applying the proximal operator, then the proximal distance
algorithm separates the parameters. Parameter separation promotes
parallelization and benefits from a fast algorithm for computing proximal
maps in one dimension. The one-dimensional algorithm is similar to but 
faster than bisection \cite{parikh13}. Finally, the objective function 
is Lipschitz with the explicit constant 
\begin{align}
L & = \sum_i \sqrt{\sum_{j \ne i}w_{ij}^2}+\|\bb\|. \label{BP_constant}
\end{align}
This assertion follows from the simple bound
\begin{align*}
|f(\bx) -f(\by)| & \le \sum_i \sum_{j \ne i} w_{ij} |x_j-y_j| +|\bb^*(\bx-\by)|\\
& \le \sum_i \sqrt{\sum_{j\ne i} w_{ij}^2} \cdot \|\bx-\by\|+\|\bb\| \cdot\|\bx-\by\|
\end{align*}
under the symmetry convention $w_{ij}=w_{ji}$. 

\begin{table}[!ht]
\begin{center}
\begin{tabular}{crrc}
\toprule
& \multicolumn{2}{c}{CPU times } \\ 
\cmidrule(r){2-3} 
Dimension\!\!\! & MM & CVX & Iterations \\ \hline 
\;\;\;\;\;\;2    & 0.038 & 0.080 &  \;\;\;\;9   \\ 
\;\;\;\;\;\;4    & 0.052 & 0.060 &  \;\;18  \\ 
\;\;\;\;\;\;8    & 2.007 & 0.050 &  200 \\ 
\;\;\;\;16   & 2.416 & 0.100 &  200 \\ 
\;\;\;\;32   & 2.251 & 0.130 &  200 \\ 
\;\;\;\;64   & 4.134 & 0.400 &  200 \\ 
\;\;128  & 0.212 & 2.980 &  \;\;32  \\ 
\;\;256  & 0.868 & 62.63 &  200 \\ 
\;\;512  & 68.27 & 1534  &  200 \\ 
1024 & 526.6 &   *   &  200 \\ 
2048 & 127.2 &   *   &  200 \\ 
4096 & 547.4 &   *   &  200 \\ 
\bottomrule 
\end{tabular} 
\end{center} 
\caption{CPU times in seconds and MM iterations until convergence for 
binary piecewise linear functions. Asterisks denote  
computer runs exceeding computer memory limits. Iterations were capped
at $200$.} 
\label{tab:mm_vs_cvx_PWLoBox}
\end{table}

Table \ref{tab:mm_vs_cvx_PWLoBox} displays the numerical results for a few typical examples.
For each dimension $d$ we filled $\bb$ with standard normal deviates
and the upper triangle of the weight matrix $\bW$ with the absolute values
of such deviates.  The lower triangle of $\bW$ was determined by symmetry.
Small values of $\bb$ often lead to degenerate solutions 
$\bx$ with all entries $0$ or $1$. To avoid this
possibility, we multiplied each entry of $\bb$ by $d$.
In the graph cut context, a degenerate solution corresponds to 
no cuts at all or a completely cut graph.  These examples depend on the schedules 
$\rho_n = \min\{1.2^n,L\}$ and $\epsilon_n = \max\{1.2^{-n},10^{-15}\}$ 
for the two tuning constants and the local Lipschitz constant (\ref{BP_constant}).

Although the MM proximal 
distance algorithm makes good progress towards the minimum in the first $100$ iterations,
it sometimes hovers around its limit without fully converging.
This translates into fickle compute times, and for this reason we
capped the number of MM iterations at $200$. For small dimensions 
MM can be much slower than CVX. Fortunately, the performance of the MM
algorithm improves markedly as $d$ increases. In all runs
the two algorithms reach the same solution after rounding
components to the nearest integer. MM also requires much less storage than CVX. 
Asterisks appear in the table where CVX demanded more 
memory than our laptop computer could deliver. \qed 

\begin{example}
Nonnegative Quadratic Programming \label{nnqp_example}
\end{example}
The proximal distance algorithm is applicable in minimizing a convex 
quad\-ratic $f(\bx) = \frac{1}{2}\bx^* \bA \bx+\bb^*\bx$
subject to the constraint $\bx \ge {\bf 0}$.  In this
nonnegative quadratic programming program,
let $\by_n$ be the projection of the current iterate $\bx_n$ 
onto $S=\mathsf{R}_+^d$. If we define the weight 
\begin{align*}
w_n & =  \frac{\rho}{\sqrt{\|\bx_n-\by_n\|^2+\epsilon}},
\end{align*}
then the next iterate can be expressed as
\begin{align*}
\bx_{n+1} & = (\bA+w_n \bI)^{-1}(w_n \by_n-\bb).
\end{align*}
The multiple matrix inversions implied by the update can be avoided
by extracting and caching the spectral decomposition $\bU^* \bD \bU$ of $\bA$
at the start of the algorithm.
The inverse $(\bA+w_n \bI)^{-1}$ then reduces to $\bU^* (\bD+w_n\bI)^{-1} \bU$.
The diagonal matrix $\bD+w_n\bI$ is obviously trivial to invert. The
remaining operations in computing $\bx_{n+1}$ collapse to matrix times 
vector multiplications. Nonnegative least squares is a special case of 
nonnegative quadratic programming.

\begin{table}[thb]
\begin{center}
\begin{tabular}{rrrrrrrrr}
\toprule
& \multicolumn{4}{c}{CPU times} & \multicolumn{4}{c}{Optima} \\ 
\cmidrule(r){2-5} \cmidrule(r){6-9} 
$d$  & MM   & CV   & MA   & YA   & MM     & CV     & MA     & YA \\ \hline 
8 & 0.97 & 0.23 & 0.01 & 0.13 & -0.0172 & -0.0172 & -0.0172 & -0.0172\\ 
16 & 0.50 & 0.24 & 0.01 & 0.11 & -1.1295 & -1.1295 & -1.1295 & -1.1295\\ 
32 & 0.50 & 0.24 & 0.01 & 0.14 & -1.3811 & -1.3811 & -1.3811 & -1.3811\\ 
64 & 0.57 & 0.28 & 0.01 & 0.13 & -0.5641 & -0.5641 & -0.5641 & -0.5641\\ 
128 & 0.79 & 0.36 & 0.02 & 0.14 & -0.7018 & -0.7018 & -0.7018 & -0.7018\\ 
256 & 1.66 & 0.65 & 0.06 & 0.22 & -0.6890 & -0.6890 & -0.6890 & -0.6890\\ 
512 & 5.61 & 2.95 & 0.26 & 0.73 & -0.5971 & -0.5968 & -0.5970 & -0.5970\\ 
1024 & 32.69 & 21.90 & 1.32 & 2.91 & -0.4944 & -0.4940 & -0.4944 & -0.4944\\ 
2048 & 156.7 & 178.8 & 8.96 & 15.89 & -0.4514 & -0.4505 & -0.4512 & -0.4512\\ 
4096 & 695.1 & 1551 & 57.73 & 91.54 & -0.4690 & -0.4678 & -0.4686 & -0.4686\\  
\bottomrule 
\end{tabular} 
\end{center} 
\caption{CPU times in seconds and optima for the nonnegative quadratic program.  Abbreviations: 
$d$ stands for problem dimension, MM for the proximal distance algorithm, CV for CVX, MA for MATLAB's \texttt{quadprog}, and YA for YALMIP.}
\label{tab:mm_vs_cvx_NQP}
\end{table}

One can estimate an approximate Lipschitz constant for this problem.
Note that $f({\bf 0}) = 0$ and that
\begin{align*}
f(\bx) & \ge  \frac{1}{2}\lambda_{\min} \|\bx\|^2-\|\bb\| \cdot \|\bx\|,
\end{align*}
where $\lambda_{\min}$ is the smallest eigenvalue of $\bA$.
It follows that any point $\bx$ with $\|\bx\|> \frac{2}{\lambda_{\min}} \|\bb\|$
cannot minimize $f(\bx)$ subject to the nonnegativity constraint.
On the other hand, the gradient of $f(\bx)$ satisfies
\begin{align*}
\|\nabla f(\bx)\| & \le  \|\bA\| \|\bx\|+\|\bb\| \le \lambda_{\max} \|\bx\|+\|\bb\|.
\end{align*}
In view of the mean-value inequality, these bounds suggest that 
\begin{align*}
L &  =  \left(\frac{2 \lambda_{\max}}{\lambda_{\min}}+1\right)\|\bb\| = \left[2\cond_2(\bA)+1 \right]\|\bb\|
\end{align*}
provides an approximate Lipschitz constant for $f(\bx)$ on the
region harboring the minimum point. This bound on $\rho$ is usually too 
large. One remedy is to multiply the bound by a deflation factor such as 0.1. 
Another remedy is to replace the covariance $\bA$ by the corresponding correlation matrix. 
Thus, one solves the problem for the preconditioned matrix $\bD^{-1}\bA \bD^{-1}$, 
where $\bD$ is the diagonal matrix whose entries are the square roots of the corresponding 
diagonal entries of $\bA$. The transformed parameters $\by = \bD \bx$ obey the same nonnegativity constraints
as $\bx$.

For testing purposes we filled a $d \times d$
matrix $\bM$  with independent standard normal deviates
and set $\bA = \bM^* \bM + \bI$. Addition of the identity
matrix avoids ill conditioning. We also filled
the vector $\bb$ with independent standard normal deviates. Our gentle tuning constant 
schedule $\epsilon_n = \max\{ 1.005^{-n}, 10^{-15} \}$ 
and $\rho_n = \min\{ 1.005^{n},0.1 \times L \}$ adjusts $\rho$ and $\epsilon$
so slowly that their limits are not actually met in practice. In any event $L$ is the a priori bound 
for the correlation matrix derived from $\bA$. 
Table \ref{tab:mm_vs_cvx_NQP} compares the performance of the MM 
proximal distance algorithm to MATLAB's \texttt{quadprog},
CVX with the SDPT3 solver, and YALMIP with the MOSEK solver. MATLAB's \texttt{quadprog}
is clearly the fastest of the four tested methods on these problems.
The relative speed of the MM algorithm improves as the problem dimension $d$ increases. 
\qed

\newpage
\begin{example}
Linear Regression under an $\ell_0$ Constraint
\end{example}
In this example the objective function is the sum of squares $\frac{1}{2}\|\by-\bX\bbeta\|^2$,
where $\by$ is the response vector, $\bX$ is the design matrix, and $\bbeta$ is the vector of
regression coefficients. The constraint set $S_k^d$ consists of those $\bbeta$ with at
most $k$ nonzero entries. Projection onto the closed but nonconvex set $S_k^d$ is
achieved by zeroing out all but the $k$ largest coordinates in absolute value. These
coordinates will be unique except in the rare circumstances of ties. The proximal
distance algorithm for this problem coincides with that of the previous problem
if we substitute $\bX^* \bX$ for $\bA$, $- \bX^*\by$ for $\bb$, $\bbeta$ for $\bx$,
and the projection operator $P_{S_k^d}$ for $P_{\mathsf{R}_+^d}$.  Better accuracy
can be maintained if the MM update exploits the singular value decomposition
of $\bX$ in forming the spectral decomposition of $\bX^*\bX$. Although the 
proximal distance algorithm carries no absolute guarantee
of finding the optimal set of $k$ regression coefficients, it is far more efficient
than sifting through all $\binom{d}{k}$ sets of size $k$. The alternative of
lasso-guided model selection must contend with strong shrinkage and a surplus of
false positives.  

\begin{table}[t]
\begin{center}
\begin{tabular}{rrrccrrrrrrr}
\toprule
$m$ & $n$ & $df$ & $tp_1$ & $tp_2$ & $\lambda\;\;\;$ & $L_1\;$ & \multicolumn{1}{r}{$L_1 / L_2$} & $T_1\;$ & $T_1 / T_2$ \\
\hline
256 & 128 & 10 & 5.97 & 3.32 & 0.143 & 248.763 & 0.868 & 0.603 & 8.098 \\ 
128 & 256 & 10 & 3.83 & 1.91 & 0.214 & 106.234 & 0.744 & 0.999 & 10.254 \\ 
512 & 256 & 10 & 6.51 & 2.88 & 0.119 & 506.570 & 0.900 & 0.907 & 6.262 \\ 
256 & 512 & 10 & 4.50 & 1.82 & 0.172 & 241.678 & 0.835 & 1.743 & 8.687 \\ 
1024 & 512 & 10 & 7.80 & 5.25 & 0.101 & 1029.333 & 0.921 & 2.597 & 5.057 \\ 
512 & 1024 & 10 & 5.54 & 2.58 & 0.138 & 507.451 & 0.881 & 8.235 & 13.532 \\ 
2048 & 1024 & 10 & 8.98 & 8.49 & 0.080 & 2047.098 & 0.945 & 15.460 & 8.858 \\ 
1024 & 2048 & 10 & 6.80 & 2.93 & 0.110 & 1044.640 & 0.916 & 34.997 & 18.433 \\ 
4096 & 2048 & 10 & 9.75 & 9.90 & 0.060 & 4086.886 & 0.966 & 89.684 & 10.956 \\ 
2048 & 4096 & 10 & 8.36 & 6.60 & 0.086 & 2045.645 & 0.942 & 166.386 & 25.821 \\ 

\bottomrule
\end{tabular}
\caption{Numerical experiments comparing MM to MATLAB's \texttt{lasso}. 
Each row presents averages over 100 independent simulations.
Abbreviations: $m$ the number of cases, $n$ the number of predictors,
$df$ the number of actual predictors in the generating model,
$tp_1$ the number of true predictors selected by MM, $tp_2$ the number
of true predictors selected by the lasso, $\lambda$ the regularization 
parameter at the lasso optimal loss, $L_1$ the optimal loss from MM,
$L_1 / L_2$ the ratio of $L_1$ to the optimal lasso loss,
$T_1$ the total computation time in seconds for MM, and 
$T_1 / T_2$ the ratio of $T_1$ to the total computation time of the lasso.}
\label{tab:mm_vs_cvx_L0_table}
\end{center}
\end{table}

Table \ref{tab:mm_vs_cvx_L0_table} compares the MM proximal distance 
algorithm to MATLAB's \texttt{lasso} function. In simulating data,
we filled $\bX$ with standard normal deviates, set all components of 
$\bbeta$ to 0 except for $\beta_i = 1/i$ for $1 \le i \le 10$,
and added a vector of standard normal deviates to $\bX \bbeta$ to determine $\by$.
For a given choice of $m$ and $n$ we ran each experiment 100 times and averaged the results.
The table demonstrates the superior speed of the lasso and
the superior accuracy of the MM algorithm as measured by optimal
loss and model selection. \qed 

\newpage
\begin{example}
Matrix Completion
\end{example}
Let $\bY =(y_{ij})$ denote a partially observed $p \times q$ matrix and $\Delta$ the 
set of index pairs $(i,j)$ with $y_{ij}$ observed.  Matrix completion 
\cite{candes09} imputes the missing entries by approximating $\bY$ 
with a low rank matrix $\bX$. Imputation relies on the singular value 
decomposition
\begin{align}
\bX & =  \sum_{i=1}^r \sigma_i \bu_i \bv_i^t , \label{svd_formula}
\end{align}
where $r$ is the rank of $\bX$, the nonnegative singular values $\sigma_i$ 
are presented in decreasing order, the left singular vectors $\bu_i$ are orthonormal, 
and the right singular vectors $\bv_i$ are also orthonormal \cite{golub96}. The set 
$R_k$ of $p \times q$ matrices of rank $k$ or less is closed. Projection onto $R_k$ 
is accomplished by truncating the sum (\ref{svd_formula}) to
\begin{align*}
P_{R_k}(\bX) & =  \sum_{i=1}^{\min\{r,k\}} \sigma_i \bu_i \bv_i^t. 
\end{align*}
When $r>k$ and $\sigma_{k+1}=\sigma_k$, the projection operator is multi-valued.

The MM principle allows one to restore the symmetry lost in the missing entries \cite{mazumder10}.  
Suppose $\bX_n$ is the current approximation to $\bX$. One simply replaces a missing entry $y_{ij}$ 
of $\bY$ for $(i,j) \not\in \Delta$ by the corresponding entry $x_{nij}$ of $\bX_n$ and adds the term $\frac{1}{2}(x_{nij}-x_{ij})^2$ to the least squares criterion 
\begin{align*}
f(\bX) & =  \frac{1}{2} \sum_{(i,j) \in \Delta} (y_{ij}-x_{ij})^2 .
\end{align*}
Since the added terms majorize 0, they create a legitimate surrogate function.
One can rephrase the surrogate by defining the orthogonal complement operator 
$P^{\perp}_{\Delta}(\bY)$ via the equation  $P^{\perp}_{\Delta}(\bY)+P_{\Delta}(\bY) = \bY$. The 
matrix $\bZ_n = P_{\Delta}(\bY)+P^{\perp}_{\Delta}(\bX_n)$ temporarily completes 
$\bY$ and yields the surrogate function $\frac{1}{2}\|\bZ_n -\bX\|_F^2$. In 
implementing a slightly modified version of the proximal distance algorithm, 
one must solve for the minimum of the Moreau function
\begin{align*}
\frac{1}{2}\|\bZ_n -\bX\|_F^2+\frac{w_n}{2}\|\bX-P_{R_k}(\bX_n)\|_F^2.
\end{align*}
The stationarity condition
\begin{align*}
{\bf 0} & =  \bX-\bZ_n+ w_n [\bX-P_{R_k}(\bX_n)]
\end{align*}
yields the trivial solution 
\begin{align*}
\bX_{n+1} & =  \frac{1}{1+w_n}\bZ_n+\frac{w_n}{1+w_n}P_{R_k}(\bX_n).
\end{align*}
Again this is guaranteed to decrease the objective function
\begin{align*}
F_\rho(\bX) & =  \frac{1}{2} \sum_{(i,j) \in \Delta} (y_{ij}-x_{ij})^2 
+\frac{\rho}{2}\dist_\epsilon(\bX,R_k)
\end{align*}
for the choice $w_n = \rho/\dist_\epsilon(\bX_n,R_k)$.  

\begin{table}[thb]
\begin{center}
\begin{tabular}{rrrrrrrrrr}
\toprule
$p$   	&	$q$   	& $\alpha$	&  rank	&	$L_1$   &	$L_1 / L_2$ &    $T_1$  & $T_1 / T_2$ \\
\hline
200 	&	250 	&	0.05	&	20	&	1598	&    	0.251	&	4.66	&		7	 \\
800 	&	1000	&	0.20	&	80	&	571949	&	 	0.253	&	131.02	&		18.1 \\
1000	&	1250	&	0.25	&	100	&	1112604	&	 	0.24	&	222.2	&		15.1 \\
1200	&	1500	&	0.15	&	40	&	793126	&	 	0.361	&	161.51	&		3.6	 \\
1200	&	1500	&	0.30	&	120	&	1569105	&	 	0.235	&	367.78	&		12.3 \\
1400	&	1750	&	0.35	&	140	&	1642661	&	 	0.236	&	561.76	&		9	 \\
1800	&	2250	&	0.45	&	180	&	2955533	&		0.171	&	1176.22	&		10.1 \\
2000	&	2500	&	0.10	&	20	&	822673	&	 	0.50	&	307.89	&		1.9	 \\
2000	&	2500	&	0.50	&	200	&	1087404	&	 	0.192	&	2342.32	&		2	 \\
5000	&	5000	&	0.05	&	30	&	7647707	&		0.664	&	1827.16	&		2	 \\
\bottomrule
\end{tabular}
\caption{Comparison of the MM proximal distance algorithm to SoftImpute. Abbreviations: $p$
is the number of rows, $q$ is the number of columns, $\alpha$ is the ratio of observed entries to total entries, $L_1$ is the optimal loss under MM, $L_2$ is the optimal loss under SoftImpute, 
$T_1$ is the total computation time (in seconds) for MM, and $T_2$ is the total computation time for SoftImpute.}
\label{tab:mm_vs_cvx_MC_table}
\end{center}
\end{table}

In the spirit of Example \ref{nnqp_example}, let us derive a
local Lipschitz constant based on the value 
$f({\bf 0}) = \frac{1}{2} \sum_{(i,j) \in \Delta} y_{ij}^2 $. The inequality
\begin{align*}
\frac{1}{2}\sum_{(i,j) \in \Delta} y_{ij}^2  & <  
\frac{1}{2}\sum_{(i,j) \in \Delta} (y_{ij}-x_{ij})^2  
= \frac{1}{2}\sum_{(i,j) \in \Delta} (y_{ij}^2-2y_{ij}x_{ij}+x_{ij}^2 )
\end{align*}
is equivalent to the inequality
\begin{align*}
2 \sum_{(i,j) \in \Delta} y_{ij}x_{ij} & < \sum_{(i,j) \in \Delta} x_{ij}^2.
\end{align*}
In view of the Cauchy-Schwarz inequality 
\begin{align*}
\sum_{(i,j) \in \Delta} y_{ij}x_{ij} & \le  
\sqrt{\sum_{(i,j) \in \Delta} y_{ij}^2}\sqrt{\sum_{(i,j) \in \Delta} x_{ij}^2}\, ,
\end{align*}
no solution $\bx$ of the constrained problem can satisfy
\begin{align*}
\sqrt{\sum_{(i,j) \in \Delta} x_{ij}^2} & > 2 \sqrt{\sum_{(i,j) \in \Delta} y_{ij}^2}\, .
\end{align*}
When the opposite inequality holds,
\begin{align*}
\|\nabla f(\bx)\|_F & = \sqrt{\sum_{(i,j) \in \Delta} (x_{ij}-y_{ij})^2} 
\le  \sqrt{\sum_{(i,j) \in \Delta} x_{ij}^2}+\sqrt{\sum_{(i,j) \in \Delta} y_{ij}^2} 
\le  3 \sqrt{\sum_{(i,j) \in \Delta} y_{ij}^2} \, .
\end{align*}
Again this tends to be a conservative estimate of the required local bound on $\rho$. 

Table \ref{tab:mm_vs_cvx_MC_table} compares the performance of the MM proximal distance algorithm 
and a MATLAB implementation of SoftImpute \cite{mazumder10}. Although the proximal distance algorithm is noticeably slower, it substantially lowers the optimal loss and improves
in relative speed as problem dimensions grow. \qed

\begin{example}
Sparse Inverse Covariance Estimation
\end{example}
The graphical lasso has applications in estimating sparse inverse covariance matrices
\cite{friedman08}. In this context, one minimizes the convex criterion 
\begin{align*}
 -\ln \det \bTheta +\tr(\bS \bTheta)+\rho \|\bTheta\|_1,
\end{align*}
where $\bTheta^{-1}$ is a $p \times p$ theoretical covariance matrix,
$\bS$ is a corresponding sample covariance matrix, and the graphical lasso
penalty $\|\bTheta\|_1$ equals the sum of the absolute values of the off-diagonal
entries of $\bTheta$. The solution exhibits both sparsity and shrinkage.
One can avoid shrinkage by minimizing 
\begin{align*}
f(\bTheta)  & =  -\ln \det \bTheta +\tr(\bS \bTheta)
\end{align*}
subject to $\bTheta$ having at most $2k$ nonzero off-diagonal entries.
Let $T_k^p$ be the closed set of $p \times p$ symmetric matrices possessing 
this property. Projection of a symmetric matrix $\bM$ onto $T_k^p$ can be 
achieved by arranging the above-diagonal entries of $\bM$ in decreasing 
absolute value and replacing all but the first $k$ of these entries by 0.  
The below-diagonal entries are treated similarly.

The proximal distance algorithm for minimizing $f(\bTheta)$ subject to the
set constraints operates through the convex surrogate
\begin{align*}
g(\bTheta \mid \bTheta_n) & = f(\bTheta)+\frac{w_n}{2}\|\bTheta-P_{T_k^p}(\bTheta_n)\|_F^2 \\
w_n & = \frac{\rho}{\sqrt{\|\bTheta_n-P_{T_k^p}(\bTheta_n)\|_F^2+\epsilon}}.
\end{align*} 
A stationary point minimizes the surrogate and satisfies
\begin{align}
{\bf 0} & =  -\bTheta^{-1} +w_n\bTheta + \bS -w_nP_{T_k^p}(\bTheta_n). \label{graphical_lasso_stationarity}
\end{align} 
If the constant matrix $\bS -w_nP_{T_k^p}(\bTheta_n)$ has spectral
decomposition $\bU_n \bD_n \bU_n^*$, then multiplying equation (\ref{graphical_lasso_stationarity})
on the left by $\bU_n^*$ and on the right by $\bU_n$ gives
\begin{align*}
{\bf 0} & =  -\bU_n^*\bTheta^{-1}\bU_n + w_n \bU_n^*\bTheta\bU_n+\bD_n.
\end{align*}
This suggests that we take $\bE = \bU_n^*\bTheta \bU_n$ to be diagonal and
require its diagonal entries $e_i$ to satisfy
\begin{align*}
0 & =  -\frac{1}{e_i}+w_n e_i+d_{ni}.
\end{align*}
Multiplying this identity by $e_i$ and solving for the positive root of the resulting quadratic yields
\begin{align*}
e_i & =  \frac{-d_{ni}+\sqrt{d_{ni}^2+4w_n}}{2w_n}.
\end{align*}
Given the solution matrix $\bE_{n+1}$, we reconstruct 
$\bTheta_{n+1}$ as $\bU_n \bE_{n+1}\bU_n^*$.  

Finding a local Lipschitz constant is more challenging in this example.
Because the identity matrix is feasible, the minimum cannot exceed
\begin{align*}
-\ln \det \bI+\tr(\bS \bI) & =  \tr(\bS)  = \sum_{i=1}^p \omega_i,
\end{align*} 
where $\bS$ is assumed positive definite with eigenvalues $\omega_i$ ordered
from largest to smallest. If the candidate matrix $\bTheta$ is positive
definite with ordered eigenvalues $\lambda_i$, then the von Neumann-Fan inequality 
\cite{borwein00} implies
\begin{align}
f(\bTheta) & \ge  -\sum_{i=1}^p \ln \lambda_i+\sum_{i=1}^p \lambda_i \omega_{p-i+1}.
\label{eigen_bound1}
\end{align}
To show that $f(\bTheta) > f(\bI)$ whenever any $\lambda_i$
falls outside a designated interval, note that
the contribution $-\ln \lambda_j + \lambda_j \omega_{p-j+1}$ to
the right side of inequality (\ref{eigen_bound1}) is bounded below by
$\ln \omega_{p-j+1}+1$ when $\lambda_j = \omega_{p-j+1}^{-1}$. Hence, 
$f(\bTheta) > f(\bI)$ whenever
\begin{align}
-\ln \lambda_i+\lambda_i \omega_{p-i+1} & >  \sum_{i=1}^p \omega_i
- \sum_{j \ne i} (\ln \omega_{p-j+1}+1). 
\label{eigen_bound2}
\end{align}
Given the strict convexity of the function
$-\ln \lambda_i+\lambda_i \omega_{p-i+1}$, equality holds in inequality (\ref{eigen_bound2})
at exactly two points $\lambda_{i \min}>0$ and $\lambda_{i \max}>\lambda_{i \min}$.
These roots can be readily extracted by bisection or Newton's method. The strict inequality
$f(\bTheta) > f(\bI)$ holds when any $\lambda_i$ falls to the left of $\lambda_{i \min}$ or to the
right of $\lambda_{i \max}$. Within the intersection of the intervals 
$[\lambda_{i \max},\lambda_{i \min}]$, the gradient of $f(\bTheta)$ satisfies
\begin{align*}
\|\nabla f(\bTheta)\|_F & \le  \|\bTheta^{-1}\|_F+\|\bS\|_F 
\le  \sqrt{\sum_{i=1}^p \lambda_i^{-2}}+\|\bS\|_F 
\le  \sqrt{\sum_{i=1}^p \lambda_{i \min}^{-2}}+\|\bS\|_F.
\end{align*}
This bound serves as a local Lipschitz constant near the optimal point. 

\begin{table}[thb]
\begin{center}
\begin{tabular}{rrrrcrccr}
\toprule
$p$ & $k_{t}$ & $k_{1}$ & $k_{2}$ & $\rho$ & \multicolumn{1}{c}{$L_1$} & $L_2 - L_1$ 
& $T_1$ & $T_1 / T_2$ \\ 
\hline
$8$ & $18$ & $14.0$ & $14.0$ & $0.00186$ & $-12.35$ & $0.01$ & $0.022$ & $43.458$ \\ 
$16$ & $42$ & $30.5$ & $28.7$ & $0.00305$ & $-25.17$ & $0.08$ & $0.026$ & $43.732$ \\ 
$32$ & $90$ & $53.5$ & $49.9$ & $0.00330$ & $-50.75$ & $0.17$ & $0.054$ & $31.639$ \\ 
$64$ & $186$ & $97.8$ & $89.3$ & $0.00445$ & $-98.72$ & $0.53$ & $0.234$ & $28.542$ \\ 
$128$ & $378$ & $191.6$ & $169.9$ & $0.00507$ & $-196.09$ & $1.14$ & $1.060$ & $18.693$ \\ 
$256$ & $762$ & $345.0$ & $304.2$ & $0.00662$ & $-369.62$ & $2.55$ & $4.253$ & $9.559$ \\ 
$512$ & $1530$ & $636.4$ & $566.8$ & $0.00983$ & $-641.89$ & $6.72$ & $19.324$ & $5.679$ \\ 
\bottomrule
\end{tabular}
\caption{Numerical results for precision matrix estimation. Abbreviations: $p$ for matrix dimension, 
$k_{t}$ for the number of nonzero entries in the true model, $k_1$ for the number of true nonzero entries 
recovered by the MM algorithm, $k_2$ for the number of true nonzero entries recovered by \texttt{glasso}, 
$\rho$ the average tuning constant for \texttt{glasso} for a given $k_t$,
$L_1$ the average loss from the MM algorithm, $L_1 - L_2$ the difference between $L_1$ and the average 
loss from \texttt{glasso}, $T_1$ the average compute time in seconds for the MM algorithm, and $T_1 / T_2$ 
the ratio of $T_1$ to the average compute time for \texttt{glasso}.}
\label{tab:mm_vs_cvx_SPM_table}
\end{center}
\end{table}

Table \ref{tab:mm_vs_cvx_SPM_table} compares the performance of the MM algorithm to
that of the R \texttt{glasso} package \cite{friedman08}.
The sample precision matrix $\bS^{-1}= \bL \bL^* + \delta \bM \bM^*$ was 
generated by filling the diagonal and first three subdiagonals of the banded
lower triangular matrix $\bL$ with standard normal deviates.  Filling
$\bM$ with standard normal deviates and choosing $\delta = 0.01$ 
imposed a small amount of noise obscuring the band nature of $\bL \bL^*$.
All table statistics represent averages over 10 runs started at 
$\bTheta = \bS^{-1}$ with $k$ equal to the true number of nonzero entries
in $\bL \bL^*$. The MM algorithm performs better in minimizing average loss and 
recovering nonzero entries.   \qed

\section{Discussion}

The MM principle offers a unique and potent perspective on high-dimensional optimization.
The current survey emphasizes proximal distance algorithms and their applications in
nonlinear programming. Our construction of this new class of algorithms relies on
the exact penalty method of Clarke \cite{clarke83} and majorization of a smooth 
approximation to the Euclidean distance to the constraint set.  Well-studied
proximal maps and Euclidean projections constitute the key ingredients of
seven realistic examples. These examples illustrate the versatility
of the method in handling nonconvex constraints, its improvement as problem
dimension increases, and the pitfalls in sending the
tuning constants $\rho$ and $\epsilon$ too quickly to their limits. 
Certainly,
the proximal distance algorithm is not a panacea for optimization problems.
For example,
the proximal distance algorithm as formulated exhibits remarkably fickle behavior on linear programming problems.
For linear programming,
we ensure numerical stability and guard against premature convergence only by great care in parameter tuning and updating. 
Nonetheless, we are sufficiently encouraged to pursue this research further, particularly in statistical applications where model
fitting and selection are compromised by aggressive penalization. 

\section{Acknowledgments}
Kenneth Lange was supported by NIH grants from the National Human Genome Research 
Institute (HG006139) and the National Institute of General Medical Sciences (GM053275). 
Kevin L. Keys was supported by a National Science Foundation Graduate Research Fellowship 
under Grant Number DGE-0707424. 
We are grateful to Hua Zhou for carefully analyzing the proximal distance algorithm in linear programming.

\end{document}